\documentclass[12pt,a4paper]{article}
\usepackage[a4paper,top=2.5cm,bottom=2.5cm,left=2.5cm,right=2.5cm]{geometry}
\usepackage{amsmath}
 \usepackage{graphicx}
\usepackage{amsthm}
\usepackage{color}
\usepackage{fancyhdr}
\usepackage{pdfpages}
\usepackage{txfonts}
\usepackage{amssymb}
\usepackage{paralist}
\usepackage{extarrows}
\usepackage{setspace}
\usepackage[misc]{ifsym}
\usepackage{authblk}
\usepackage{braket} 
\usepackage[colorlinks=true,citecolor=blue,linkcolor=blue,anchorcolor=blue,urlcolor=blue]{hyperref}
\allowdisplaybreaks
\numberwithin{equation}{section}
\newtheorem{theorem}{Theorem}[section]
\newtheorem{lemma}{Lemma} [section]
\newtheorem{proposition}[theorem]{Proposition}
\newtheorem{corollary}[theorem]{Corollary}

\newtheorem{remark}{Remark}[section]

\begin{document}
\title{Multiplicity of Positive Solutions of Nonlinear Elliptic Equation  with Gradient Term}

\author[a]{Fei Fang}
\author[b]{Zhong Tan}
\author[a]{Huiru Xiong$^{(\textrm{\Letter})}$\thanks{$^{(\textrm{\Letter})}$Corresponding Author\quad E-mail: xionghuiru@126.com }
}
\affil[a]{School of Mathematics and Statistics, Beijing Technology and Business University, Beijing 100048, China}
\affil[b]{School of Mathematical Sciences, Xiamen University, 361005, P. R. China}
\maketitle

\noindent \textbf{\textbf{Abstract:}}  In this paper,  we consider the following nonlinear elliptic equation with gradient term:
\[
\left\{ \begin{gathered}
 - \Delta u - \frac{1}{2}(x \cdot \nabla u) + (\lambda a(x)+b(x))u =  \beta u^q +u^{2^*-1}, \hfill \\
  0<u \in {H_K^{1}(\mathbb{R}^N)}, \hfill \\
\end{gathered}  \right .
\]
where $\lambda, \beta \in (0,\infty), q \in (1,2^*-1), 2^* = 2N/(N-2), N\geq3, a(x), b(x): \mathbb{R}^N \to \mathbb{R}$ are continuous functions, and $a(x)$ is nonnegative  on $\mathbb{R}^N$. When $\lambda$  is large enough, we prove the existence and multiplicity of positive solutions to the equation.



\section{Introduction}
In this paper,  we  consider the following equation:
\begin{equation}\label{e1}
\left\{ \begin{gathered}
Lu:=- \Delta u - \frac{1}{2}(x \cdot \nabla u)  = f(x,u), \hfill \\
  0<u \in {H_K^{1}(\mathbb{R}^N)}. \hfill \\
\end{gathered}  \right .
\end{equation}
The operator $L$ is closely related to the self-similar solutions of the heat equation, which was studied by Escobedo and Kavian in \cite{c8} (also see \cite{c25,c22}).
 The operator $L$ appears in the process of looking for the self-similar solutions
 $$v(t,x)=t^{-1/(p-2)}u(t^{-1/2}x)$$
 of the heat equation
 \[v_t - \Delta v =|v|^{p-2}v.\]
  Escobedo and Kavian expressed the operator $L$ as the form of a divergence, that is,
\[Lu:=- \Delta u - \frac{1}{2}(x \cdot \nabla u)= - \frac{1}{K}\nabla  \cdot (K\nabla u) ,\]
where $K(x): = {e^{|x{|^2}/4}},$ so the operator $L$  has a variational structure.
They also equipped the operator $L$ with a weighted Sobolev space and proved related embedding theorem in  \cite{c8}. On the other hand, assume that  $(M,g)$ is a Riemannian manifold, $f$ is a smooth function on $M$, and the weight volume of $M$ is of the form $\mathrm{e}^{-f} dV_g$. The operator  $L'$ is defined by
$$
L'u:=\Delta_g u-\langle\nabla_g f, \nabla_g u\rangle,
$$
where $\nabla_g$ and $\Delta_g$ denote the gradient operator and Laplace operator on $M$ respectively. It is easy to see that the operator  $L'=L$ when  $M=\mathbb{R}^N, g$ is the unit matrix, $\nabla_g f=x$.  The operator  $L'$is an important research object in geometric analysis, which is closely related to  Ricci solution and  Ricci flow. The reader is referred to the paper \cite{c2,c3,c4,c5,c6} for more studies on the properties and applications of the operator  $L'$.

In recent years, the equation \eqref{e1} has been studied and some results haved been obtained. In 2004, if $f(x,u)=\frac{1}{p-1}u+u^p$, Naito  \cite{c20} obtained at least two positive self-similar solutions. In 2007, Catrina et al.  \cite{c21} established the existence of positive solutions  when  considered the case $f(x,u)=u^{2^*-1}+\lambda|x|^{\alpha -2}u$, where $2^*=2N/(N-2)$, $\alpha \geq 2$. In 2014, Furtado et al. in \cite{c22} proved the existence of at least two nonnegative nontrivial solutions for the equation when $f(x,u)=a(x)|u|^{q-2}u+b(x)|u|^{p-2}u$, with $1<q<2<p\leq2^*$ and certain conditions on $a(x)$ and $b(x)$. In 2017, Li et al.  \cite{c23} obtained a ground state solution for the equation \eqref{e1}. In 2019, Figueiredo investigated the case of changing sign solutions for the equation in \cite{c24}.

Now we assume that $f(x,u)=\beta u^q +u^{2^*-1}-(\lambda a(x)+b(x))u$ and study following equation:
\begin{equation}\label{a1}
\left\{ \begin{gathered}
 - \Delta u - \frac{1}{2}(x \cdot \nabla u) + (\lambda a(x)+b(x))u =  \beta u^q +u^{2^*-1}, \hfill \\
  0<u \in {H_K^{1}(\mathbb{R}^N)}. \hfill \\
\end{gathered}  \right .
\end{equation}
When the equation \eqref{a1} does not contain the gradient term, it becomes the following elliptic equation:
\begin{equation}\label{a2}
\left\{ \begin{gathered}
 - \Delta u  + (\lambda a(x)+b(x))u =  \beta u^q +u^{2^*-1}, \hfill \\
  0<u \in {H_K^{1}(\mathbb{R}^N)}. \hfill \\
\end{gathered}  \right .
\end{equation}
Claudianor et al. in \cite{c13} proved the multiplicity of positive solutions for the equation \eqref{a2}. We adopt a similar proof strategy as in \cite{c13} (also see \cite{c12,c7}) to establish the multiplicity of positive solutions for the equation \eqref{a1} with the gradient term. In order to obtain our conclusions, we make the following assumptions:

$(a_1)$  $a(x)\in C(\mathbb{R}^N,\mathbb{R})$ and $a(x)\geq 0$ for all $x \in \mathbb{R}^N$. The set $\text{int}\ a^{-1}(0) := \Omega$ is a nonempty bounded open set with smooth boundary, consisting of $k$ connected components $\Omega_j$, where $j\in\{1,\dots,k\}$. Moreover, we have $d(\Omega_i,\Omega_j)>0$ for $i\ne j$. In other words,
\[\Omega=\Omega_1\cup\Omega_2\cup\cdots\cup\Omega_k,\]
and $a^{-1}(0) := \overline{\Omega}$.

$(b_1)$ $b(x)\in C(\mathbb{R}^N,\mathbb{R})$ and there exists a positive constant $M_1$ such that
\begin{equation} \label{e3}
|b(x)|\leq M_1, \forall x \in \mathbb{R}^N .
\end{equation}

$(a_2)$ There exists a positive constant $M_0$ such that $a(x)$ and $b(x)$ verify
\begin{equation}\label{e2}
0<M_0\leq a(x)+b(x), \forall x \in \mathbb{R}^N.
\end{equation}

For any $j\in \{1,\dots,k\}$, we fix a bounded open subset $\Omega_j^{'}$ with smooth boundary satisfying:

(i) $\overline{\Omega_j}\subset\Omega_j^{'},$

(ii) $\overline{\Omega_j^{'}}\cap \overline{\Omega_l^{'}}=\emptyset$ for all $l\ne j.$\\
Additionally, we also fix a nonempty subset $\Gamma\subset\{1,\dots,k\}$, and define the sets
\[{\Omega _\Gamma } = \bigcup\limits_{j \in \Gamma } {{\Omega _j}}, \quad
{\Omega _\Gamma^{'} } = \bigcup\limits_{j \in \Gamma } {{\Omega _j^{'}}}.\]

The main theorem of this paper is given below.
\begin{theorem}\label{t1}
Let $a,b$ satisfy $(a_1),(a_2)$ and $(b_1)$. For any nonempty subset $\Gamma\subset\{1,\dots,k\}$, there exist constants $\beta^*>0$ and $\lambda^*=\lambda^*(\beta^*)$, such that for any $\beta\geq \beta^*$ and $\lambda\geq \lambda^*$, the equation \eqref{a1} has a family of positive solutions $\{u_{\lambda}\}$ with the following property:
For any sequence $\lambda_n\to\infty$, there exists a subsequence $\{\lambda_{n_i}\}$ such that $u_{\lambda_{n_i}}$ strongly converges  in ${H_K^{1}(\mathbb{R}^N)}$ to $u(x)=0$  for $x\ne\Omega_\Gamma$,  and the restriction $u|_{\Omega_j}$ is a least energy solution of the problem below for all $j\in\Gamma$:
\[\left\{ \begin{gathered}
   - \Delta u - \frac{1}{2}(x \cdot \nabla u) + b(x)u = \beta {u^q} + {u^{2^* - 1}}  \in \Omega_j,\hfill \\
  u > 0 \quad {\rm in}\ \Omega_j ,\hfill \\
  u = 0 \quad {\rm on} \ \partial \Omega_j. \hfill \\
\end{gathered}  \right.\]
\end{theorem}

\begin{corollary}
Under the assumptions of Theorem \ref{t1}, there exist constants $\beta^*>0$ and $\lambda^*=\lambda^*(\beta^*)$, such that for $\beta\geq \beta^*$ and $\lambda\geq \lambda^*$, equation \eqref{a1} has at least $2^k-1$ positive solutions.
\end{corollary}

Furtado et al.\cite{c26} studied the equation with a nonlinear term $f(x,u)=\lambda |x|^{\beta}|u|^{q-2}u+|u|^{2^*-2}u$ in the critical growth case, where $\lambda>0$, $2\leq q<2^*$, $\beta=(\alpha-2)(2^*-q)/(2^*-2)$, and $\alpha\geq 2$. For $2<q<2^*$, Furtado et al. obtained one positive solution, and for $q=2$, they obtained a sign-changing solution. Catrina et al. also studied the case of a critical growth nonlinear term $f(x,u)=\lambda |x|^{\alpha-2}u+|u|^{2^*-1}$ in \cite{c21}, and proved the existence of at least two positive solutions. In this paper, we also consider the case of a nonlinear term with critical growth and obtain at least $2^k-1$ positive solutions.

The structure of this article consists of five parts. In Section 2, we will introduce the basic concepts and relevant lemmas. In Sections 3 and 4, we will prove the (PS) condition and the critical value of the functional. In Section 5, we will prove Theorem \ref{t1}.

\section{Preliminaries}
Define the set
\[L_K^q({\mathbb{R}^N}): = \left\{ {u:{\mathbb{R}^N} \to \mathbb{R}:\int_{{\mathbb{R}^N}} {K(x)|u{|^q}dx}  < \infty } \right\},\]
and  equip it with the following norm:
\[\:|u|_{K,q}: = {\left( {\int_{{\mathbb{R}^N}} {K(x)|u{|^q}} dx} \right)^{\frac{{\text{1}}}{q}}},\ q\in[1,\infty)\]
and
\[|u|_{K,\infty } := \mathop {{\text{ess }}\sup }\limits_{x \in {\mathbb{R}^N}} |u(x)|,\ q=\infty.\]
Therefore, the space $L_K^\infty(\mathbb{R}^N)=L^\infty(\mathbb{R}^N)$ is compatible to the other spaces(see\cite[page 880]{c27}), that is,
\[\lim_{q \to \infty } |u|_{K,q}=|u|_{K,\infty}, \ u \in L_K^1(\mathbb{R}^N)\cap L^\infty(\mathbb{R}^N).\]

We further define the spaces
\[H_K^1(\mathbb{R}^N): = \left\{ u:\mathbb{R}^N \to \mathbb{R}: \int_{\mathbb{R}^N} {K(x)(|\nabla u|^2+|u|^2)dx}   < \infty \right\}\]
and
\[H_{K,\lambda }^1(\mathbb{R}^N): = \left\{ u \in H_K^1(\mathbb{R}^N):\int_{\mathbb{R}^N} {K(x)(\lambda a(x) + b(x))u^2dx}  < \infty \right\},\]
equipped with the following norms:
\[\| u \|_K: = \left( \int_{\mathbb{R}^N} {K(x)(|\nabla u|^2+|u|^2)dx}\right)^{\frac{1}{2}},\]
\[\| u \|_{K,\lambda}: = \left( \int_{\mathbb{R}^N} {K(x)(|\nabla u|^2+(\lambda a(x)+b(x))|u|^2)dx}\right)^{\frac{1}{2}}.\]
We denote the dual space of $H_{K,\lambda }^1$ by $H_{K,\lambda }^*$, and $\langle\cdot,\cdot\rangle:H_{K,\lambda }^*\times H_{K,\lambda }^1$ represents the duality pairing. For $\lambda\geq1$, it can be observed that $\left(H_{K,\lambda }^1(\mathbb{R}^N),{\left\| \cdot \right\|_{K,\lambda}}\right)$ is a Hilbert space, and the embedding $H_{K,\lambda }^1(\mathbb{R}^N)\hookrightarrow H_K^1(\mathbb{R}^N)$ is continuous.

Let $u\in H_{K,\lambda }^1(\mathbb{R}^N)$ is a weak solution of equation \eqref{a1}, if for any $\varphi\in H_{K,\lambda }^1({\mathbb{R}^N})$, there is
\[\int_{\mathbb{R}^N} {K(x)(\nabla u\cdot\nabla \varphi + (\lambda a(x) + b(x))u\varphi)dx}  - \beta \int_{\mathbb{R}^N} K(x)u^q\varphi dx - \int_{\mathbb{R}^N} {K(x)u^{2^*-1}\varphi} =0,\]
\[\langle I^{\prime}(u),\varphi\rangle=\int_{\mathbb{R}^N} {K(x)(\nabla u\cdot\nabla \varphi + (\lambda a(x) + b(x))u\varphi)dx}  - \beta \int_{\mathbb{R}^N} K(x)u^q\varphi dx - \int_{\mathbb{R}^N} {K(x)u^{2^*-1}\varphi},\]
where
\[I(u): = \frac{1}{2}\int_{{\mathbb{R}^N}} {K(x)(|\nabla u{|^2} + (\lambda a(x) + b(x)){u^2})dx}  - \frac{\beta }{{q + 1}}\int_{{\mathbb{R}^N}} {K(x){{({u_ + })}^{q + 1}}} dx - \frac{1}{{{2^*}}}\int_{{\mathbb{R}^N}} {K(x){{({u_ + })}^{{2^*}}}} dx,\]
\[u_+(x)=\text{max}\{u(x),0\}.\]
It is easy to see that a nonnegative weak solution to the equation \eqref{a1} is the critical point of the function $I:H_{K,\lambda }^1({\mathbb{R}^N})\to\mathbb{R}$.

Similarly, for an open set $\Theta\subset\mathbb{R}^N$, we can define
\[H_{K,\lambda }^1(\Theta): = \left\{ u \in H_K^1(\Theta):\int_{\Theta} {K(x)(\lambda a(x) + b(x))u^2dx}  < \infty \right\}\]
and
\[\| u \|_{K,\lambda,\Theta}: = \left( \int_{\Theta} {K(x)(|\nabla u|^2+(\lambda a(x)+b(x))|u|^2)dx}\right)^{\frac{1}{2}}.\]
Analogously, we use $|u|_{K,q,\Theta}$ to represent the norm of the space $L_K^q(\Theta)$.
According to assumption $(a_2)$ with  \eqref{e2}, we can obtain
\[M_0|u|_{K,2,\Theta}^2\leq \int_{\Theta} {K(x)(|\nabla u|^2+(\lambda a(x)+b(x))|u|^2)dx},\ \forall\ u \in H_{K,\lambda }^1(\Theta), \lambda \geq 1,\]
which is equivalent to
\[|u|_{K,\lambda,\Theta}^2\geq M_0|u|_{K,2,\Theta}^2,\ \forall\ u \in H_{K,\lambda }^1(\Theta), \lambda \geq 1.\]

\begin{proposition}[Embedding Theorem\ \cite{c8} ]\label{p5}
For all $1<q\leq{2^*} = 2N/(N-2)$, the embedding $H_K^{1}({\mathbb{R}^N}) \hookrightarrow L_K^q({\mathbb{R}^N})$ is continuous. For all $1<q<{2^*}$, the embedding $H_K^{1}({\mathbb{R}^N}) \hookrightarrow L_K^q({\mathbb{R}^N})$ is compact.
\end{proposition}

\begin{proposition}[Concentration-Compactness Principle \cite{c4}]\label{l3}
Let $\{u_n\}\subset H_K^1({\mathbb{R}^N})$ be a bounded sequence such that $u_n\rightharpoonup u$ in $L_K^{2^*}(\mathbb{R}^N)$. If there exist measures $\nu$ and $\mu$, and a subsequence of $\{u_n\}$ such that $|u_n|_{K,2^*}^{2^*}\rightharpoonup\nu$ and $|\nabla u_n|_{K,2}^2\rightharpoonup\mu$, then there exist sequences $\{x_n\}\subset {\mathbb{R}^N}$ and $\{u_n\}\subset [0,\infty)$ satisfying
\[|u_n|_{K,2^*}^{2^*}\rightharpoonup|u|_{K,2^*}^{2^*}+\sum\limits_{i = 1}^\infty  {{\nu _i}{\delta _{{x_i}}}} \equiv \nu , \]
\[\sum\limits_{n = 1}^\infty  {\nu_n^{2/2^*}}<\infty, \quad \mu(x_n)\geq S\nu_n^{2/2^*},\ \forall\ n\in \mathbb{N}, \]
where $\delta_i$ is the Dirac measure and $S$ is the best Sobolev constant of the embedding $H_K^1({\mathbb{R}^N}) \hookrightarrow L_K^{2^*}({\mathbb{R}^N})$, given by
\[S:=\mathop {\inf }\limits_{x \in H_K^1({\mathbb{R}^N})\setminus\{0\}}\left\{\int_{\mathbb{R}^N} {K(x)(|\nabla u{|^2+|u|^2)}dx} \ \bigg| \int_{\mathbb{R}^N} {K(x)|u{|^{2^*}}dx}=1\right\}.\]
\end{proposition}

\begin{lemma}[\cite{c12}]\label{l1}
There exist constants $\delta_0, \nu_0 > 0$ with $\delta_0\approx 1$ and $\nu_0\approx 0$ such that ,  for all open sets $\Theta\subset\mathbb{R}^N$,
\begin{equation}\label{e12}
\delta_0||u||_{K,\lambda,\Theta}^2\leq||u||_{K,\lambda,\Theta}^2
-\nu_0|u|_{K,2,\Theta}^2,\ \forall\ u\in H_{K,\lambda }^1(\Theta),\lambda\geq1.
\end{equation}
\end{lemma}

\section{(PS) Condition and Research on Energy Levels}

In this section, we adapt some argumentation approaches of Pino and Felmer \cite{c7}, Ding and Tanaka \cite{c12}, and Claudianor et al. \cite{c13} to prove several lemmas.

Let us define a function $h:\mathbb{R}\to\mathbb{R}$ as follows:
\[h(t) = \left\{
\begin{array}{ll}
\beta {t^q} + {t^{{2^*} - 1}}, & \quad t \geq 0, \\
0, & \quad t \leq 0,
\end{array}
\right.\]
and fix a positive constant $a$ verifying $h(a)/a=\nu_0$, where $\nu_0>0$ is the constant provided in Lemma \ref{l1}.
Additionally, we introduce two functions $f$ and $F:\mathbb{R}\to\mathbb{R}$, which play vital roles in the subsequent content.
\[f(t) = \left\{ \begin{gathered}
  0,t \leq 0, \hfill \\
  h(t),t \in [0,a], \hfill \\
  \nu_0t,t \geq a, \hfill \\
\end{gathered}  \right.\]
\[F(t) = \int_0^t {f(\tau )d\tau }  = \left\{ \begin{gathered}
  0,t \leq0, \hfill \\
  \frac{\beta }{{q + 1}}{t^{q + 1}} + \frac{1}{{{2^*}}}{t^{{2^*}}},t \in [0,a], \hfill \\
  \frac{\beta }{{q + 1}}{a^{q + 1}} + \frac{1}{{{2^*}}}{a^{{2^*}}} + \frac{1}{2}{\nu _0}({t^2} - {a^2}),t \geq a. \hfill \\
\end{gathered}  \right.\]
Using the set \ ${\Omega _\Gamma^{'} }$, we consider the function
\[\chi_\Gamma(x) = \left\{ \begin{gathered}
  1,x\in {\Omega _\Gamma^{'} }, \hfill \\
  0,x\notin {\Omega _\Gamma^{'} }, \hfill \\
\end{gathered}  \right.\]
\[g(x,t)=\chi_\Gamma(x)h(t)+(1-\chi_\Gamma(x))f(t),\]
\[G(x,t) = \int_0^t {g(x,\tau)d\tau}= \chi_\Gamma(x)H(t)+(1-\chi_\Gamma(x))F(t),\]
where
\[H(t) = \int_0^t {h(\tau )d\tau }. \]
We use  $\Phi_\lambda:H_{K,\lambda }^1(\mathbb{R}^N)\to\mathbb{R}$ to present that
\[\Phi_\lambda(u)= \frac{1}{2}\int_{{\mathbb{R}^N}} {K(x)(|\nabla u{|^2} + (\lambda a(x) + b(x)){u^2})dx}  - \int_{{\mathbb{R}^N}} {K(x)G(x,u) dx} .\]
It is easy to know that  $\Phi_\lambda\in C^1(H_{K,\lambda }^1(\mathbb{R}^N),\mathbb{R})$, the critical point of $\Phi_\lambda$ is a nonnegative weak solution to the following equation,
\begin{equation}\label{e4}
 - \Delta u - \frac{1}{2}(x \cdot \nabla u) + (\lambda a(x)+b(x))u = g(x,u).
\end{equation}
Note that the positive solution of the above equation is related to the positive solution of equation \eqref{a1}. If $u\in H_{K,\lambda }^1(\mathbb{R}^N)\to\mathbb{R}$ is a positive solution of equation \eqref{e4}, then it can be verified that $u(x)\leq a$  in $\mathbb{R}^N\setminus {\Omega _\Gamma^{'} }$ is a positive solution of equation \eqref{a1}.

\begin{remark}
Based on the definitions of $f$ and $F$, we assume that the (PS) sequences are nonnegative.
\end{remark}

\begin{lemma}\label{l4}\label{l5}
For $\lambda\geq1$, any (PS) sequence $\{u_n\} \subset H_{K,\lambda }^1(\mathbb{R}^N)$ on the functional $\Phi_\lambda$ is uniformly bounded, i.e.,  there exists  constant $m(c)$ and $M(c)$ that is independent of $\lambda\geq 1$, such that
\[m(c)\leq\mathop {\lim }\limits_{n \to \infty } \inf ||u_n||_{K,\lambda}^2\leq\mathop {\lim }\limits_{n \to \infty } \sup ||u_n||_{K,\lambda}^2\leq M(c).\]
Moreover, if $c>0$, then $m(c)>0$.
\end{lemma}

\begin{proof}
Let $\{u_n\}\subset H_{K,\lambda }^1(\mathbb{R}^N)$ be a $\text{(PS)}_c$ sequence, then we have
\[\Phi_\lambda(u_n)\to c,\quad \Phi_\lambda^{\prime}(u_n)\to 0.\]
For $n$ sufficiently large, by the above expression, we have
\[\Phi_\lambda(u_n)-\frac{1}{{q + 1}}\langle \Phi_\lambda^{\prime}(u_n),u_n\rangle=c+o(1)
+\varepsilon_n||u_n||_{K,\lambda},\]
where $\varepsilon_n\to 0.$ Therefore,
\begin{equation}\label{e31}
\left(\frac{1}{2}-\frac{1}{{q + 1}}\right)||u_n||_{K,\lambda}^2-\int_{\mathbb{R}^N\setminus\Omega_\Gamma^{'}} {K(x)\left[F(u_n)-\frac{1}{{q + 1}}f(u_n)u_n\right] dx}=c+o(1)+\varepsilon_n||u_n||_{K,\lambda}.
\end{equation}
We note that
\[F(t) - \frac{1}{{q + 1}}f(t)t = \left\{ {\begin{array}{*{20}{l}}
{0,t \le 0,}\\
{ \left( {\frac{1}{{{2^*}}} - \frac{1}{{q + 1}}} \right){t^{{2^*}}},t \in [0,a],}\\
{\frac{\beta}{{q + 1}} {a^{q + 1}} + \frac{1}{{{2^*}}}{a^{{2^*}}} + \left( {\frac{1}{2} - \frac{1}{{q + 1}}} \right){\nu _0}{t^2} - \frac{1}{2}{\nu _0}{a^2},t \ge a.}
\end{array}} \right.\]
Hence,
\[F(t)-\frac{1}{{q + 1}}f(t)t\leq\left(\frac{1}{2}- \frac{1}{{q + 1}}\right)\nu_0(t^2-a^2)\leq\left(\frac{1}{2}- \frac{1}{{q + 1}}\right)\nu_0t^2, t \in \mathbb{R},\]
and we have
\[\left(\frac{1}{2}-\frac{1}{{q + 1}}\right)\left(||u_n||_{K,\lambda}^2
-\nu_0|u_n|_{K,2}^2\right)\leq c+o(1)
+\varepsilon_n||u_n||_{K,\lambda}.\]
Using Lemma \ref{l1}, we have
\[\delta_0\left(\frac{1}{2}-\frac{1}{{q + 1}}\right)||u_n||_{K,\lambda}^2
\leq c+o(1)+\varepsilon_n||u_n||_{K,\lambda}.\]
Thus, $||u_n||_{K,\lambda}$ is bounded as $n\to \infty$ and
\[\mathop {\lim }\limits_{n \to \infty } \sup ||u_n||_{K,\lambda}^2\leq M(c):=\left(\frac{1}{2}-\frac{1}{{q + 1}}\right)^{-1}\delta_0^{-1}c.\]
On the other hand, it follows from \eqref{e31} that
\[\left(\frac{1}{2}-\frac{1}{{2^*}}\right)||u_n||_{K,\lambda}^2-\int_{\mathbb{R}^N\setminus\Omega_\Gamma^{'}} {K(x)\left[F(u_n)-\frac{1}{{2^*}}f(u_n)u_n\right] dx}>c+o(1)+\varepsilon_n||u_n||_{K,\lambda},\]
so
\[\left(\frac{1}{2}-\frac{1}{{2^*}}\right)||u_n||_{K,\lambda}^2>c+o(1)+\varepsilon_n||u_n||_{K,\lambda},\]
\[\mathop {\lim }\limits_{n \to \infty } \inf ||u_n||_{K,\lambda}^2\geq m(c):=\left(\frac{1}{2}-\frac{1}{{2^*}}\right)^{-1}c.\]
This shows that $\{u_n\}$ is uniformly bounded in $H_{K,\lambda }^1(\mathbb{R}^N)$.
\end{proof}

Next, for each fixed $j\in\Gamma$, we denote by  $c_j$ the minimax level of the mountain-pass theorem associated with the  function \ $I_j:H_K^1(\Omega_j)\to \mathbb{R}$, given by
\begin{equation}\label{e21}
I_j(u) = \frac{1}{2}\int_{\Omega_j} {K(x)(|\nabla u{|^2} + b(x){u^2})dx} - \frac{\beta }{{q + 1}}\int_{\Omega_j} {K(x){{({u_ + })}^{q + 1}}} dx - \frac{1}{{{2^*}}}\int_{\Omega_j} {K(x){{({u_ + })}^{{2^*}}}} dx.
\end{equation}
It can be seen that the critical points of $I_j$ are weak solutions to the following problem:
\begin{equation}\label{e27}
\left\{ \begin{gathered}
- \Delta u - \frac{1}{2}(x \cdot \nabla u) + b(x)u = \beta {u^q} + {u^{2^* - 1}}, \quad \text{in}\ \Omega_j,\hfill \\
u > 0, \quad \text{in}\ \Omega_j, \hfill \\
u = 0, \quad \text{on}\ \partial \Omega_j. \hfill \\
\end{gathered} \right.
\end{equation}

\begin{lemma}\label{l2}
There exists ${\beta ^*} > 0$ such that for any $\beta\geq\beta^*$, we have
\[{c_j} \in \left( {0,\left( {\frac{1}{2} - \frac{1}{{q + 1}}} \right)\frac{{{S^{N/2}}}}{{k + 1}}} \right),\ \forall \ j \in \{ 1, \cdots ,k\} .\]
\end{lemma}
\begin{proof}
For any $j \in \{ 1, \cdots ,k\}$, we fix a nonnegative function $\varphi_j\in H_K^1(\Omega_j)\setminus\{0\}$. We note that there exists $t_{\beta,j}\in(0,+\infty)$ such that
\[c_j\leq I_j(t_{\beta,j}\varphi_j)=\mathop {\max }\limits_{t \geq 0} {I_j}(t{\varphi _j}).\]
Therefore, the following equation holds:
\[\int_{\Omega_j} {K(x)(|\nabla \varphi_j{|^2}  + b(x){|\varphi_j|^2})dx} = \beta t_{\beta,j}^{q-1}\int_{\Omega_j} {K(x){{\varphi_j}^{q + 1}}} dx +t_{\beta,j}^{2^*-2} \int_{\Omega_j} {K(x){{\varphi_j}^{{2^*}}}} dx .\]
Above equation implies that
\[{t_{\beta ,j}} \leq {\left[ {\frac{\int_{\Omega_j} {K(x)(|\nabla \varphi_j{|^2}  + b(x){|\varphi_j|^2})dx}}{{\beta \int_{{\Omega _j}} {K(x){\varphi _j}^{q + 1}} dx}}} \right]^{1/(q - 1)}},\]
\[t_{\beta ,j}\to 0, \ \beta \to +\infty.\]
Using the above limits, we have
\[I_j(t_{\beta,j}\varphi_j)\to 0,\ \beta \to +\infty. \]
Thus, it can be seen that there exists $\beta^*>0$ such that
\[{c_j} < \left( {\frac{1}{2} - \frac{1}{{q + 1}}} \right)\frac{{{S^{N/2}}}}{{k + 1}},\ \forall \ j \in \{ 1, \cdots ,k\},\ \forall\ \beta \in [\beta^*,+\infty) .\]
\end{proof}

\begin{remark}
In particular, the above lemma implies that
\begin{equation}\label{e23}
\sum\limits_{j = 1}^k  {c_j\in \left( {0,\left( {\frac{1}{2} - \frac{1}{{q + 1}}} \right)S^{N/2}} \right)}.
\end{equation}
\end{remark}

\begin{lemma}\label{p1}
For each $\lambda  \geq 1$ and $c \in \left( {0,\left( {\frac{1}{2} - \frac{1}{q+1}} \right){S^{N/2}}} \right)$, any $(\rm{PS})_c$ sequence $\{u_n\}\subset H_{K,\lambda }^1(\mathbb{R}^N)$ on the functional $\Phi_\lambda$ has a strongly convergent subsequence in $H_{K,\lambda }^1(\mathbb{R}^N)$.
\end{lemma}

\begin{proof}
Let $\{u_n\}\subset H_{K,\lambda }^1(\mathbb{R}^N)$ be a $\text{(PS)}_c$ sequence. According to Lemma $\ref{l5}$, we know that the sequence $\{u_n\}$ is bounded in $H_{K,\lambda }^1(\mathbb{R}^N)$. Therefore, we can assume that
\[u_n \rightharpoonup u \quad \text{in}\ H_{K,\lambda}^1(\mathbb{R}^N) \ \text{and} \ H_{K}^1(\mathbb{R}^N),\]
\[u_n \to u \quad \text{in}\ L_{K}^p(\mathbb{R}^N),\forall\ p\in[2,2^*).\]
Since $\{u_n\}$ is a bounded $\text{(PS)}_c$ sequence, let $\varphi_n(x)=\eta(x)u_n(x),$
we have
\[\langle\Phi_\lambda^{\prime}(u_n),\varphi_n \rangle=\langle\Phi_\lambda^{\prime}(u_n),\eta u_n \rangle=o(1),\]
where the cut-off function $\eta \in C^{\infty}(\mathbb{R}^N)$ satisfies
\[\eta (x) = \left\{ \begin{gathered}
  1,\forall x \in B_R^c(0), \hfill \\
  0,\forall x \in B_{R/2}(0), \hfill \\
\end{gathered}  \right.\]
\[\eta(x)\in [0,1],\quad \Omega_{\Gamma}^{'}\subset B_{R/2}(0),\]
where $ B_R^c(0)=\{x\in\mathbb{R}^N:|x|\geq R\}$. Using the argument method of Lemma 1.1 in \cite{c7} (also see \cite{c1}), we know that for every $\varepsilon>0$, there exists $R>0$ such that
\begin{equation}\label{e8}
\int_{\{x\in{\mathbb{R}^N}:|x|\geq R\}} {K(x)(|\nabla u_n{|^2} + (\lambda a(x) + b(x)){u_n^2})dx} \leq \varepsilon,\ n\in \mathbb{N}.
\end{equation}

Applying Proposition \ref{l3} to the sequence $\{u_n\}$, we obtain a sequence $\{\nu_n\}$ such that $\nu_n=0$ for all $n\in \mathbb{N}$. Therefore,
\begin{equation}\label{e20}
u_n \to u \quad \text{in}\ L_{K,\text{loc}}^{2^*}(\mathbb{R}^N).
\end{equation}
In fact, once we prove that $\{u_n\}$ is a $\text{(PS)}_c$ sequence, for every $\phi\in C_0^{\infty}(\Omega)$, we can multiply both sides of equation \eqref{e4} by $u_n\phi$, integrate by parts, and obtain
\begin{align}\label{e5}
&\int_{{\mathbb{R}^N}} {K(x)|\nabla u_n{|^2}\phi dx} +\int_{{\mathbb{R}^N}} {K(x)\nabla u_n \nabla\phi dx} +\int_{{\mathbb{R}^N}} {K(x)(\lambda a(x) + b(x)){u_n^2}\phi dx} \\ \nonumber
&= \int_{{\mathbb{R}^N}} {K(x)g(x,u_n)u_n\phi} dx +o(1) .
\end{align}
If $\{x_n\}$ is the sequence given in Proposition \ref{l3}, let $\Phi_{\varepsilon}=\Phi(x-x_n)/{\varepsilon}$, $x\in \mathbb{R}^N$, $\varepsilon >0$, where $\Phi\in C_0^{\infty}(\mathbb{R}^N,[0,1])$ verifying $\Phi\equiv1$ on $B_1(0)$, $\Phi\equiv0$ on $B_2^c(0)$, and $|\nabla \Phi|\leq 2$. Considering $\phi=\Phi_\varepsilon$ in equation \eqref{e5}, for all $n\in\mathbb{N}$, we can use the method in \cite{c14} to show that $\mu(x_n)\leq \nu_n$. If $\nu_n>0$, combining with Proposition \ref{l3}, we obtain
\begin{equation}\label{e7}
\nu_n\geq S^{N/2},\ \forall\ n\in \mathbb{N}.
\end{equation}
Thus, it can be seen that $\{\nu_n\}$ is finite.

Next, we will prove that for all $n\in \mathbb{N}$, $\nu_n=0$. Again, using the fact that $\{u_n\}$ is a $\text{(PS)}_c$ sequence, we have
\[I(u_n)-\frac{1}{q+1}\langle I^{\prime}(u_n),u_n\rangle=c+o(1).\]
Therefore, we have
\begin{align*}
\left( \frac{1}{2}-\frac{1}{q+1}\right)\int_{{\mathbb{R}^N}} {K(x)|\nabla u_n{|^2} dx} &+  \left( \frac{1}{2}-\frac{1}{q+1}\right)\int_{{\mathbb{R}^N}} {K(x)(\lambda a(x)+b(x))u_n^2 dx}\\
&+ \int_{{\mathbb{R}^N}} {K(x)\left[ \frac{1}{q+1}g(x,u_n)u_n-G(x,u_n) \right] dx}=c+o(1).
\end{align*}
Since
\[ \int_{{\mathbb{R}^N}} {K(x)(\lambda a(x)+b(x))u_n^2 dx}
+ \int_{{\mathbb{R}^N}} {K(x)\left[ \frac{1}{q+1}g(x,u_n)u_n-G(x,u_n) \right] dx}\geq 0,\]
we can conclude that
\[\left( {\frac{1}{2} - \frac{1}{{q + 1}}} \right)\int_{{\mathbb{R}^N}} {K(x)|\nabla {u_n}{|^2}dx}  \leq c + {o}(1).\]
Then,
\begin{equation}\label{e6}
\left( {\frac{1}{2} - \frac{1}{{q + 1}}} \right)\mu(x_n)\leq c,\ \forall \ n\in \mathbb{N}.
\end{equation}
Since $\mu(x_n)\geq S\nu_n^{2/2^*}$, if there exists  $\nu_n>0$ for some $n\in \mathbb{N}$, from \eqref{e7} and \eqref{e6}, we obtain the inequality
\[c\geq \left( {\frac{1}{2} - \frac{1}{{q + 1}}} \right)S^{N/2},\]
which is a contradiction. Therefore, for all $n\in \mathbb{N}$, we have $\nu_n=0$, that is, the \eqref{e20} is established.
From  \eqref{e8} and \eqref{e20}, we can conclude that
\[\int_{{\mathbb{R}^N}} {K(x)g(x,u_n)u_ndx}\to \int_{{\mathbb{R}^N}} {K(x)g(x,u)udx},\ n\to \infty.\]
This means
 \[u_n\to u,\ \text{in}\ H_{K,\lambda }^1(\mathbb{R}^N).\]
\end{proof}

A sequence \ $\{u_n\}\subset H_{K}^1(\mathbb{R}^N)$, called $\text{ (PS)}_{\infty,c}$, is one that satisfies.
\[\left\{ \begin{gathered}
  {u_n} \in H_{K,{\lambda _n}}^1(\mathbb{R}^N), \hfill \\
  {\lambda _n} \to \infty ,\ n \to \infty , \hfill \\
  {\Phi _{{\lambda _n}}}({u_n}) \to c,\ {\lambda _n} \to \infty , \hfill \\
  {||\Phi^{\prime}_{\lambda _n}}({u_n})||_K \to 0,\ {\lambda _n} \to \infty.  \hfill \\
\end{gathered}  \right.\]

\begin{lemma}\label{p3}
Let $\{u_n\}$ be a $\text{(PS)}_{\infty,c}$ sequence with $c \in \left( {0,\left( {\frac{1}{2} - \frac{1}{q+1}} \right){S^{N/2}}} \right)$. Then, for some subsequence given by $\{u_n\}$, there exists $u\in H_K^1(\mathbb{R}^N)$ such that
 \[u_n\rightharpoonup u,\ {\rm in}\ H_K^1(\mathbb{R}^N).\]
 Moreover,

$\rm(i)$ $u\equiv 0$ in $\mathbb{R}^N\setminus\Omega_\Gamma$ and $u|_{\Omega_j}$ is a nonnegative solution of
\begin{equation}\label{e11}
\left\{ \begin{gathered}
   - \Delta u - \frac{1}{2}(x \cdot \nabla u)+b(x)u = \beta |u|^{q-1}u + {|u|^{2^* - 2}u},\ {\rm in}\ \Omega_j,\hfill \\
  u = 0, \ {\rm on} \ \partial \Omega_j, \hfill \\
\end{gathered}  \right.
\end{equation}
where $j\in \Gamma$.

$\rm(ii)$ $u_n$ converges to $u$ in a stronger sense, i.e.,
\[||u_n-u||_{K,\lambda_n}\to 0.\]
Therefore,
\[u_n\to u,\ {\rm in}\ H_K^1(\mathbb{R}^N).\]

$\rm(iii)$ As $\lambda_n \to \infty$, $u_n$ satisfies:
\[{\lambda_n}\int_{{\mathbb{R}^N}} {K(x)a(x)u_n^2dx}\to 0,\]
\[||u_n||_{K,\lambda_n,\mathbb{R}^N\setminus\Omega_\Gamma}^2 \to 0,\]
\[||u_n||_{K,\lambda_n,\Omega_j^{'}}^2  \to \int_{\Omega_j} {K(x)(|\nabla u{|^2}  + b(x){u^2})dx} ,j\in \Gamma.\]
\end{lemma}

\begin{proof}
According to Lemma \ref{l4}, there exists a positive constant $M>0$ such that
\[||u_n||_{K,\lambda_n}\leq M,\ \forall \ n\in \mathbb{N}.\]
Therefore, $\{u_n\}$ is a bounded sequence in $H_K^1(\mathbb{R}^N)$. For a subsequence still denoted by $\{u_n\}$, we can assume that there exists $u\in H_K^1(\mathbb{R}^N)$ such that
\[u_n\rightharpoonup u,\ \text{in} \ H_K^1(\mathbb{R}^N),\]
\[u_n(x)\to u(x),\ \text{a.e.} \ \mathbb{R}^N.\]
Using a similar argument as in the proof of Lemma \ref{p1}, we obtain
\begin{equation}\label{e9}
u_n\to u,\ \text{in} \ H_K^1(\mathbb{R}^N).
\end{equation}

To prove (i), we fix the set $C_m=\{x\in\mathbb{R}^N:\ a(x)\geq \frac{1}{m}\}$. Then
\[\int_{C_m} {K(x)u_n^2dx}\leq \frac{m}{\lambda_n}\int_{{\mathbb{R}^N}} {\lambda_nK(x)a(x)u_n^2dx},\]
That is,
\[\int_{C_m} {K(x)u_n^2dx}\leq \frac{m}{\lambda_n} ||u_n||_{K,\lambda_n}^2.\]
Using Fatou's lemma in the above inequality, this implies
\[\int_{C_m} {K(x)u^2dx}=0,\ \forall \ m\in \mathbb{N}.\]
Therefore, we have $u(x)=0$ on $\bigcup\limits_{m = 1}^\infty  {{C_m}}  = \mathbb{R}^N\setminus\overline{\Omega}$. We can assert that $u|_{\Omega_j}\in H_K^1(\Omega_j)$ for all $j\in\{1,\cdots,k\}$.

Once we have shown that for all $\varphi\in C_0^{\infty}(\Omega_j)$, as $n\to \infty$, we have $\langle \Phi_{\lambda_n}^{\prime}(u_n),\varphi \rangle\to 0$, then from \eqref{e9}, we have
\begin{equation}\label{e10}
\int_{\Omega_j} {K(x)(\nabla u \nabla \varphi + b(x)u\varphi) dx} - \int_{\Omega_j} {K(x)g(x,u)\varphi dx}=0.
\end{equation}
In other words, for all $j\in\{1,\cdots,k\}$, $u|_{\Omega_j}$ is a solution of the equation \eqref{e11}.

For each $j\in\{1,\cdots,k\}\setminus\Gamma$, we let $\varphi=u|_{\Omega_j}$ in \eqref{e10}, we have
\[\int_{\Omega_j} {K(x)(|\nabla u|^2 + b(x)u^2) dx} - \int_{\Omega_j} {K(x)f(u)u dx}=0,\]
That is,
\[||u||_{K,\lambda,\Omega_j}^2-\int_{\Omega_j} {K(x)f(u)u dx}=0.\]
For all $t\in \mathbb{R}$, we have $f(t)t\leq \nu_0t^2$. Using \eqref{e12}, we have
\[\delta_0||u||_{K,\lambda,\Omega_j}^2\leq||u||_{K,\lambda,\Omega_j}^2
-\nu_0|u|_{K,2,\Omega_j}^2\leq ||u||_{K,\lambda,\Omega_j}^2-\int_{\Omega_j} {K(x)f(u)u dx}=0.\]
Therefore, for $j\in\{1,\cdots,k\}\setminus\Gamma$, we have $u=0$ in $\Omega_j$. This verifies (i).

For (ii), we have
\begin{align*}
||u_n-u||_{K,\lambda_n}^2&-\int_{\mathbb{R}^N\setminus\Omega_\Gamma^{'}} {K(x)(f(u_n)-f(u))(u_n-u) dx}\\
&-\int_{\Omega_\Gamma^{'}} {K(x)(h(u_n)-h(u))(u_n-u) dx}\\
&=\langle \Phi_{\lambda_n}^{\prime}(u_n),(u_n-u) \rangle-\langle \Phi_{\lambda_n}^{\prime}(u),(u_n-u) \rangle.
\end{align*}
Using the equality
\[\int_{\Omega_\Gamma^{'}} {K(x)(h(u_n)-h(u))(u_n-u) dx}=o(1),\]
\[\langle \Phi_{\lambda_n}^{\prime}(u),(u_n-u) \rangle=\int_{\Omega_\Gamma} {K(x)[\nabla u \nabla (u_n-u) + a(x)u(u_n-u) ]dx} - \int_{\Omega_\Gamma} {K(x)f(u)(u_n-u)dx}=o(1),\]
and the inequality
\[|\langle \Phi_{\lambda_n}^{\prime}(u_n),(u_n-u) \rangle|\leq ||\Phi_{\lambda_n}^{\prime}(u_n)||_K\left(||u_n||_{K,\lambda_n}
+||u||_{K,\lambda_n}\right) =o(1),\]
We have
\[||u_n-u||_{K,\lambda_n}^2-\int_{\mathbb{R}^N\setminus\Omega_\Gamma^{'}} {K(x)(f(u_n)-f(u))(u_n-u)dx}=o(1).\]
Using equation \eqref{e12}, $u\equiv 0$ on $\mathbb{R}^N\setminus\Omega_\Gamma^{'}$, and the above estimate, we obtain
\[||u_n-u||_{K,\lambda_n}^2\to 0,\ n\to \infty.\]

To prove (iii), from equation \eqref{e2}, we have
\[\lambda_n\int_{\mathbb{R}^N} {K(x)a(x)u_n^2dx}\leq C||u_n-u||_{K,\lambda_n}^2.\]
Therefore,
\[\lambda_n\int_{\mathbb{R}^N} {K(x)a(x)u_n^2dx}\to 0,\ n\to \infty.\]
\end{proof}

To establish the uniform boundedness of $\{u_n\}$ in $L^\infty_K$, we need the following two propositions.

\begin{proposition}[\cite{c15,c13}]\label{p2}
Let $b$ be a nonnegative measurable function, and let $g:\mathbb{R}^N\times \mathbb{R}_+\to \mathbb{R}_+$ satisfy the following inequality. For every nonnegative function $v\in H_K^1(\mathbb{R}^N)$, there exists a function $h\in L_K^{N/2}(\mathbb{R}^N)$ such that
\[g(x,v(x))\leq (h(x)+C_g)v(x),\ \forall\ x\in \mathbb{R}^N.\]
If $v\in H_K^1(\mathbb{R}^N)$ is a weak solution of the equation
\[- \Delta v - \frac{1}{2}(x \cdot \nabla v)+b(x)v=g(x,v),\]
then $v\in L_K^p(\mathbb{R}^N)$ for all $2\leq p<\infty$. Moreover, there exists a positive constant $C_p=C(p,C_g,h)$ such that
\[|v|_{K,p}\leq C_p||v||_{K}.\]
If $\{v_k\},\{b_k\}$, and $\{h_k\}$ satisfy the above assumptions, and $h_k\to h$ in $L_K^{N/2}(\mathbb{R}^N)$, then the sequence $C_{p,k}=C(p,C_g,h_k)$ is bounded.
\end{proposition}

\begin{lemma}\label{p4}
Assume that $b$ is a set as in Proposition \ref{p2}, $q>N/2$, and for every nonnegative function $v\in H_K^1(\mathbb{R}^N)$, there exists $h\in L_K^q(\mathbb{R}^N)$ such that
\[g(x,v(x))\leq h(x)v(x),\ \forall\ x\in \mathbb{R}^N.\]
If $v$ is a nonnegative weak solution of the equation
\[- \Delta u - \frac{1}{2}(x \cdot \nabla u) +b(x)v=g(x,v),\]
then there exists $C=C(q,|h|_{K,q})>0$ such that
\[|v|_{K,\infty}\leq C||v||_{K}.\]
Moreover, if $\{v_k\},\{b_k\}$, and $\{h_k\}$ satisfy the above assumptions, and $|h_k|_{K,q}$ is bounded, then the sequence $C_k=C(q,|h_k|_{K,q})$ is bounded.
\end{lemma}

\begin{proof}
We prove this lemma using Moser iteration and the methods in \cite{c1,c9,c16}.

For each $n\in\mathbb{N}$ and $\alpha >1$ such that $v\in L_K^{2\alpha q_1}(\mathbb{R}^N)$. Let $A_n=\{x\in\mathbb{R}^N:|v|^{\alpha-1}\leq n\}$, $B_n=\mathbb{R}^N\setminus A_n$, and define the function $v_n$ as follows,
\[v_n := \begin{cases} v|v{|^{2(\alpha  - 1)}}, & \text{on } A_n, \\ {n^2}v, & \text{on } B_n. \end{cases}\]
Once we prove that $v_n\in H_K^1(\mathbb{R}^N)$, we have
\[\int_{\mathbb{R}^N} {K(x)(\nabla v \nabla v_n + b(x)vv_n) dx} = \int_{\mathbb{R}^N} {K(x)g(x,v)v_n dx}.\]
Consider $q_1=q/(q-1)$ and $r>2q_1$,
\[\omega_n := \begin{cases} v|v{|^{\alpha  - 1}}, & \text{on } A_n, \\ {n}v, & \text{on } B_n. \end{cases}\]
According to the proof Lemma 4.1 in \cite{c1} (or see \cite{c9}), we have
\begin{equation}\label{e13}
|v|_{K,r\alpha}\leq {\alpha}^{1/\alpha}(S_r|h|_{K,q})^{1/2\alpha}|v|_{K,2\alpha q_1}.
\end{equation}
Now, we will prove the estimate for the $L_K^\infty$ norm.

(i) Fix $\chi=r/(2q_1)>1$ and $\alpha=\chi$, we have $2q_1\alpha=r$. The inequality \eqref{e13} can be rewritten as
\begin{equation}\label{e14}
|v|_{K,r\chi}\leq {\chi}^{1/\chi}(S_r|h|_{K,q})^{1/(2\chi)}|v|_{K,r}.
\end{equation}

(ii) Consider $\alpha=\chi^2$, we have $2q_1\alpha=r\chi$. Therefore, by (i) and \eqref{e13}, we obtain
\begin{equation}\label{e15}
|v|_{K,r\chi^2}\leq {\chi}^{2/\chi^2}(S_r|h|_{K,q})^{1/(2\chi^2)}|v|_{K,r\chi}.
\end{equation}
Based on equations \eqref{e14} and \eqref{e15}, we have

\begin{equation}\label{e16}
|v|_{K,r\chi^2}\leq {\chi}^{{1/\chi}+{2/\chi^2}}(S_r|h|_{K,q})^{({1/\chi}+1/\chi^2)/2}|v|_{K,r}.
\end{equation}

(iii) Choosing $\alpha=\chi^3$, we have $2q_1\alpha=r\chi^2$. Therefore, from (ii) and equation \eqref{e13}, we can obtain
\begin{equation}\label{e17}
|v|_{K,r\chi^3}\leq {\chi}^{3/\chi^3}(S_r|h|_{K,q})^{1/2\chi^3}|v|_{K,r\chi^2}.
\end{equation}
Using equations \eqref{e16} and \eqref{e17}, we have
\begin{equation}\label{e18}
|v|_{K,r\chi^3}\leq {\chi}^{{1/\chi}+{2/\chi^2}+{3/\chi^3}}(S_r|h|_{K,q})^{({1/\chi}+{1/\chi^2}+{1/\chi^3})/2}|v|_{K,r}.
\end{equation}
Repeating the above procedure for each $m\in\mathbb{N}$, we have the following inequality:
\begin{equation}\label{e19}
|v|_{K,r\chi^m}\leq {\chi}^{{1/\chi}+{2/\chi^2}+{3/\chi^3}+\cdots+{m/\chi^m}}
(S_r|h|_{K,q})^{({1/\chi}+{1/\chi^2}+{1/\chi^3}+\cdots+{1/\chi^m})/2}|v|_{K,r}.
\end{equation}
Because
\[\sum\limits_{m = 1}^\infty  {\frac{m}{{{\chi ^m}}}}  = \frac{1}{{\chi  - 1}},\quad \frac{1}{2}\sum\limits_{m = 1}^\infty  {\frac{1}{{{\chi ^m}}}}  = \frac{1}{{2(\chi  - 1)}}.\]
From equation \eqref{e19}, we can conclude that
\[|v|_{K,r\chi^m}\leq C|v|_{K,r},\]
where $C={\chi}^{1/(\chi-1)}(S_r|h|_{K,q})^{1/2(\chi-1)}.$ Letting $m\to \infty$, we finally have
\[|v|_{K,\infty}\leq C|v|_{K,r}.\]
\end{proof}

\begin{lemma}\label{l8}
Let $\{u_\lambda\}$ be a family of positive solutions to equation \eqref{e4} satisfying
\[\mathop {\sup\limits_{\lambda  \geq 1}  \{ {\Phi _\lambda }({u_\lambda })\} } < \left( {\frac{1}{2} - \frac{1}{{q + 1}}} \right){S^{N/2}}.\]
Then there exists $\lambda^*>0$ such that
\[|u_\lambda|_{K,\infty,\mathbb{R}^N\setminus\Omega_\Gamma^{'}}\leq e,\ \forall\ \lambda \geq \lambda^*.\]
Therefore, for $\lambda\geq \lambda^*$, $u_\lambda$ is a positive solution to equation \eqref{e1}.
\end{lemma}

\begin{proof}
Let $\{\lambda_n\}$ be a sequence, $\lambda_n\to\infty$, and define $u_n(x)=u_{\lambda_n}(x)$. Then $u_{\lambda_n}(x)$ is a bounded sequence of positive solutions to equation \eqref{e4}. Using Lemma \ref{p3}, we have $u_n\to u$ in $H_K^1(\mathbb{R}^N)$, where $u$ is the weak limit of $u_n$ in $H_K^1(\mathbb{R}^N)$. Furthermore, since there exists a constant $C>0$ such that
\[g(x,u_n)\leq u_n+Cu_n^{2^*-1}\leq (1+e_n(x))u_n,\]
where $e_n(x)=C|u_n|^{2^*-2}$ and converges to $u^{2^*-2}$ in $L_K^{N/2}(\mathbb{R}^N)$. Using Proposition \ref{p2}, we know that the sequence $\{|u_n|_{K,r}\}$ is uniformly bounded for every $r>1$. Letting $r>2^*$, we can write equation \eqref{e4} as
\[ - \Delta u - \frac{1}{2}(x \cdot \nabla u) +(\lambda_n a(x)+b(x)-\nu_0)u_n =\tilde{g}(x,s):=g(x,u_n)-\nu_0u_n\in\mathbb{R}^N.\]
Note that
\[\tilde{g}(x,u_n)\leq Cu_n^{2^*-1}=e_n(x)u_n.\]
We can verify that $e_n(x)=Cu_n^{2^*-2}\in L_K^q(\mathbb{R}^N)$, where $q=r/(2^*-2)$ and $q>N/2$. Lemma \ref{p4} ensures that for some $C_0>0$,
\[|u_n|_{K,\infty}\leq C_0,\forall\ n\in\mathbb{R}.\]

Now let $v_n(x) = u_{\lambda_n}(\varepsilon_nx+\bar{x}_n)$, where $\varepsilon_n^2 = 1/\lambda_n$ and $\{\bar{x}_n\} \subset \partial\Omega_\Gamma^{'}$. Without loss of generality, we assume $\bar{x}_n \to \bar{x} \in \partial\Omega_\Gamma^{'}$. We can obtain $|v_n|_{K,\infty} \leq C_0$,
\[ - \Delta v_n - \frac{1}{2}(x \cdot \nabla v_n) +(a(\varepsilon_nx+\bar{x}_n)+\varepsilon_n^2 b(\varepsilon_nx+\bar{x}_n))v_n =  \varepsilon_n^2g(\varepsilon_nx+\bar{x}_n,v_n),\]
and
\[|g(\varepsilon_nx+\bar{x}_n,v_n)|\leq |v_n|+C|v_n|^{2^*-1}.\]
This implies the existence of $C_1 > 0$ such that
\[||v_n||_{K,C^2(B_1(0))}\leq C_1,\ \forall\ n\in \mathbb{N}.\]
The above estimate indicates that the weak limit $v$ of the sequence $\{v_n\} \subset H_K^1(\mathbb{R}^N)$ belongs to $C^1(B_1(0))$, and
\[v_n\to v\in C^1(B_1(0)),\ n\to\infty.\]
Assuming by contradiction that there exists $\eta > 0$ such that
\[u_{\lambda_n}(\bar{x}_n)\geq\eta,\ \forall\ n\in \mathbb{N},\]
then we have
\[u_n(0)\geq\eta,\ \forall\ n\in \mathbb{N}.\]
Therefore, in $B_1(0)$, $v\ne0$.

On the other hand, the function $v$ satisfies the equation
\[- \Delta v - \frac{1}{2}(x \cdot \nabla v)+a(\bar{x})v=0\in\mathbb{R}^N.\]
This implies that $v\equiv 0$, which contradicts the fact that $v\ne0$ in $B_1(0)$. Therefore, there exists $\lambda^*>0$ such that
\[|u_\lambda|_{K,\infty,\partial\Omega_\Gamma^{'}}\leq e,\ \forall\ \lambda\geq\lambda^*.\]
Through a similar proof process to theorem 0.1 in \cite{c7}, we have
\[|u_\lambda|_{K,\infty,\mathbb{R}^N\setminus\Omega_\Gamma^{'}}\leq e,\ \forall\ \lambda\geq\lambda^*.\]
\end{proof}

\section{Critical Value of the Functional $\Phi_\lambda$}
For any $\lambda \geq 1$ and $j \in \Gamma$, we define $\Phi_{\lambda,j}: H_K^1(\Omega_j^{'}) \to \mathbb{R}$ as follows:
\[\Phi_{\lambda,j}(u) = \frac{1}{2}\int_{\Omega_j^{'}} {K(x)(|\nabla u|^2 + (\lambda a(x) + b(x))u^2)dx}  - \frac{\beta }{{q + 1}}\int_{\Omega_j^{'}} {K(x){{(u_+)}^{q + 1}}} dx - \frac{1}{{2^*}}\int_{\Omega_j^{'}} {K(x){{(u_+)}^{2^*}}} dx.\]

We know that the critical points of $\Phi_{\lambda,j}$ are weak solutions of the following elliptic equation with Neumann boundary conditions,
\[\left\{ \begin{gathered}
- \Delta u - \frac{1}{2}(x \cdot \nabla u) + (\lambda a(x)+b(x))u  =  \beta u^q +u^{2^*-1}\in\Omega_j^{'} , \hfill \\
  0<u \in \Omega_j^{'}, \hfill \\
  \frac{\partial u}{\partial \eta}, \ {\rm on}\ \partial\Omega_j^{'}. \hfill \\
\end{gathered}  \right .\]

It is known that $\Phi_{\lambda,j}$ satisfies the mountain-pass geometry condition. We denote the related minimax level associated with the functional as $c_{\lambda,j}$, defined as
\[c_{\lambda ,j} := \inf_{\gamma  \in \Gamma_{\lambda ,j}} \max_{t \in [0,1]} \Phi_{\lambda ,j}(\gamma (t)),\]
where
\[\Gamma_{\lambda,j}=\left\{\gamma\in C\left([0,1],H_K^1(\Omega_j^{'})\right)\ \big|\ \gamma(0)=0,\Phi_{\lambda,j}(\gamma(1))<0\right\}.\]

Since $\beta$ is very small, by referring to \cite{c18,c17,c21,c26}, we can know that there exist two nonnegative functions $w_j\in H_K^1(\Omega_j)$ and $w_{\lambda,j}\in H_K^1(\Omega_j^{'})$ satisfying
\[I_j(w_j)=c_j,\quad I_j^{\prime}(w_j)=0, \]
\[\Phi_{\lambda,j}(w_{\lambda,j})=c_{\lambda,j},\quad \Phi_{\lambda,j}^{\prime}(w_{\lambda,j})=0,\]
where $I_j$ is defined as in equation \eqref{e21}.

Let $R>1$ be chosen such that
\[\bigg|I_j\left(\frac{1}{R}w_j\right)\bigg|<\frac{1}{2}c_j,\ \forall\ j\in\Gamma,\]
\[|I_j(Rw_j)-c_j|\geq 1,\ \forall\ j\in\Gamma.\]
From the definition of $c_j$, it is standard to prove the equality
\begin{equation}\label{e22}
\max_{s \in [1/R^2,1]} I_j(sRw_j)=I_j(w_j)=c_j,\ \forall\ j\in\Gamma.
\end{equation}
The choice of the interval $[1/R^2,1]$ is made for the subsequent proof.

Let us reconsider the set $\Gamma=\{1,\cdots,l\}$, where $l\leq k$. We define
\[{[1/{R^2},1]^l} = \underbrace {[1/{R^2},1] \times  \cdots  \times [1/{R^2},1]}_{l\text{ times}},\]
\[\gamma_0:[1/{R^2},1]^l\to\bigcup_{j \in \Gamma } {H_K^1}({\Omega_j})\subset {H_K^1}({\Omega_\Gamma^{'}}), \]
\begin{equation}\label{e25}
\gamma_0(s_1,s_2,\cdots,s_l)(x)=\sum_{j = 1}^l {s_j}R{w_j}(x),
\end{equation}
\[b_{\lambda,\Gamma}=\inf_{\gamma  \in \Gamma_*} \max_{({s_1}, \cdots ,{s_l}) \in {[1/{R^2},1]^l}} {\Phi _\lambda }(\gamma ({s_1}, \cdots ,{s_l})),\]
where
\[\Gamma_{*}=\left\{\gamma\in C\left([1/R^2,1]^l,H_K^1(\Omega_\Gamma^{'})\setminus\{0\}\right)\ \big|\ \gamma=\gamma_0,\text{on}\ \partial([1/R^2,1]^l)\right\}.\]
We observe that $\gamma_0\in\Gamma_*$, so $\Gamma_*\ne \emptyset$, and $b_{\lambda,\Gamma}$ is well-defined.

\begin{lemma}\label{e26}
For any $\gamma\in\Gamma_*$, there exists $(t_1,\cdots,t_l)\in[1/R^2,1]^l$ such that
\[\langle\Phi_{\lambda,j}^{\prime}(\gamma(t_1,\cdots,t_l)),
\gamma(t_1,\cdots,t_l)\rangle=0,\ j\in\{1,\cdots,l\}.\]
\end{lemma}

\begin{proof}
For a given $\gamma\in\Gamma_*$, we consider a mapping $\tilde{\gamma}:[1/{R^2},1]^l\to \mathbb{R}^l$ defined as
\[\tilde{\gamma}(s_1,s_2,\cdots,s_l)
=\left(\Phi_{\lambda,1}^{\prime}(\gamma)(\gamma),\cdots,
\Phi_{\lambda,l}^{\prime}(\gamma)(\gamma)\right),\]
where
\[\Phi_{\lambda,j}^{\prime}(\gamma)(\gamma)=
\langle\Phi_{\lambda,j}^{\prime}(\gamma(s_1,\cdots,s_l)),
\gamma(s_1,\cdots,s_l)\rangle,\ \forall\ j \in \Gamma.\]
For $(s_1,s_2,\cdots,s_l)\in\partial([1/R^2,1]^l)$, we have
\[\gamma(s_1,s_2,\cdots,s_l)=\gamma_0(s_1,s_2,\cdots,s_l).\]
Using equation \eqref{e22}, we obtain
\[\langle\Phi_{\lambda,j}^{\prime}(\gamma_0(s_1,\cdots,s_l)),
\gamma_0(s_1,\cdots,s_l)\rangle=0,\]
which implies
\[s_j=\frac{1}{R},\ \forall\ j\in\Gamma.\]
Therefore, $(0,\dots,0)\notin \tilde{\gamma}(\partial([1/R^2,1]^l)).$ By some algebraic manipulation, we obtain the following topological degree,
\[\text{deg}(\tilde{\gamma},[1/R^2,1]^l,(0,\cdots,0))=(-1)^l.\]
Hence, by the property of topological degree, there exists $(t_1,\cdots,t_l)\in(1/R^2,1)^l$ such that
\[\langle\Phi_{\lambda,j}^{\prime}(\gamma(t_1,\cdots,t_l)),
\gamma(t_1,\cdots,t_l)\rangle=0,\ j\in\{1,\cdots,l\}.\]
\end{proof}

In the following, we define ${c_\Gamma } := \sum\limits_{j = 1}^l {{c_j}}$. It plays a crucial role in the proof of Theorem \ref{t1}. We will analyze the relationship between $\sum\limits_{j = 1}^l {{c_{\lambda,j}}},b_{\lambda,\Gamma}$, and $c_\Gamma$, where we need the condition $c_\Gamma\in \left(0, \left( {\frac{1}{2} - \frac{1}{{q + 1}}} \right)S^{N/2}\right)$.

\begin{lemma}\label{l6}
(i) $ \sum\limits_{j = 1}^l {{c_{\lambda,j}}}\leq b_{\lambda,\Gamma}\leq c_\Gamma,\ \forall\ \lambda\geq 1.$

(ii) $ \Phi_\lambda(\gamma(s_1,s_2,\cdots,s_l))<c_\Gamma,\ \forall\ \lambda\geq 1, \gamma \in \Gamma_*,(s_1,s_2,\cdots,s_l)\in\partial([1/R^2,1]^l).$
\end{lemma}

\begin{proof}
We use a similar proof strategy as in Proposition 4.2 of \cite{c19}.

(i) From \eqref{e25}, for $\gamma_0\in \Gamma_*$, we have
\[b_{\lambda,\Gamma}\leq\mathop {\max }\limits_{({s_1}, \cdots ,{s_l}) \in {{[1/{R^2},1]}^l}} {\Phi _\lambda }(\gamma_0 ({s_1}, \cdots ,{s_l}))=\mathop {\max }\limits_{({s_1}, \cdots ,{s_l}) \in {{[1/{R^2},1]}^l}} \sum \limits_{j=1}^l{I_j(sRw_j)}= \sum \limits_{j=1}^l{c_j}=c_\Gamma.\]
Fix $(t_1,t_2,\cdots,t_l)\in[1/R^2,1]^l$ as in Lemma \ref{e26}. By the definition of $c_{\lambda,j}$, we have
\[c_{\lambda,j}=\text{inf}\left\{\Phi_{\lambda,j}\ \big|\ u\in H_K^1(\Omega_j^{'})\setminus\{0\},\langle\Phi_{\lambda,j}^{\prime}(u),
u\rangle=0\right\},\]
\[\Phi_{\lambda,j}(\gamma(t_1,\cdots,t_l))\geq c_{\lambda,j},\ \forall \ j\in\Gamma.\]
On the other hand, since
\[\Phi_{\lambda,\mathbb{R}^N\setminus\Omega_\Gamma^{'}}(u)\geq 0,\ \forall \ u \in H_K^1({\mathbb{R}^N\setminus\Omega_\Gamma^{'}}),\]
we have
\[\Phi_{\lambda}(\gamma(s_1,\cdots,s_l))\geq \sum \limits_{j=1}^l \Phi_{\lambda,j}(\gamma(s_1,\cdots,s_l)).\]
Therefore,
\[\mathop {\max }\limits_{({s_1}, \cdots ,{s_l}) \in {{[1/{R^2},1]}^l}} {\Phi _\lambda }(\gamma ({s_1}, \cdots ,{s_l})\geq {\Phi _\lambda }(\gamma ({t_1}, \cdots ,{t_l}))\geq \sum \limits_{j=1}^lc_{\lambda,j}.\]
By the definition of $b_{\lambda,\Gamma}$, we obtain
\[b_{\lambda,\Gamma}\geq \sum \limits_{j=1}^lc_{\lambda,j}.\]

(ii) Because for any $(s_1,s_2,\cdots,s_l) \in \partial([1/R^2,1]^l)$, we have
\[\gamma(s_1,s_2,\cdots,s_l)=\gamma_0(s_1,s_2,\cdots,s_l)\ \text{on}\ \partial([1/R^2,1]^l),\]
so
\[\Phi_{\lambda}(\gamma_0(s_1,\cdots,s_l))=\sum \limits_{j=1}^l{I_j(s_jRw_j)}.\]
Furthermore, for all $j\in\Gamma,$ we have ${I_j(s_jRw_j)}\leq c_j$. For some $j_0\in\Gamma,s_{j_0}\in\{1/R^2,1\}$, we have ${I_{j_0}(s_{j_0}Rw_{j_0})}\leq c_{j_0}/2.$ Therefore, for some $\varepsilon>0$, we have
\[\Phi_{\lambda}(\gamma_0(s_1,\cdots,s_l))\leq c_\Gamma-\varepsilon\]
for some $\varepsilon>0$.
\end{proof}

\begin{corollary}\label{co1}
(i) As $\lambda \to \infty$, $b_{\lambda,\Gamma}\to c_\Gamma$.

(ii) When $\lambda$ is large, $b_{\lambda,\Gamma}$ is a critical value of $\Phi_\lambda$.
\end{corollary}

\begin{proof}
(i) For all $\lambda\geq 1$ and each $j$, we have $0<c_{\lambda,j}\leq c_j$. Using a similar proof strategy as in Lemma \ref{p3}, we can show that as $\lambda \to \infty$, $c_{\lambda,j}\to c_j$. Therefore, from Lemma \ref{l6}, we conclude that as $\lambda \to \infty$, $b_{\lambda,\Gamma}\to c_\Gamma$.

(ii) From (i) and equation \eqref{e23}, we choose a sufficiently large $\lambda$ such that
\[b_{\lambda,\Gamma}\simeq c_\Gamma\in\left(0, \left( {\frac{1}{2} - \frac{1}{{q + 1}}} \right)S^{N/2}\right).\]
Lemma \ref{p1} implies that any $\text{(PS)}_{b_{\lambda,\Gamma}}$ sequence of the functional $\Phi_\lambda$ has a strongly convergent subsequence in $H_{K,\lambda}(\mathbb{R}^N)$. Using this fact, we can conclude from the argument of the deformation lemma that for $\lambda \geq 1$, $b_{\lambda,\Gamma}$ is a critical value of $\Phi_\lambda$.
\end{proof}

\section{Proof of the Main Theorem}
To prove Theorem \ref{t1}, we need to find a positive solution $u_\lambda$ that approximates the a least-energy solution in each $\Omega_j$ when $\lambda$ is large, and vanishes elsewhere as $\lambda \to \infty$. To do this, we will prove two lemmas that, combined with the estimates made in the above section, can establish the validity of Theorem \ref{t1}.

Let
\[M := 1 + \sum\limits_{j = 1}^k {\sqrt {{{\left( {\frac{1}{2} - \frac{1}{{q + 1}}} \right)}^{ - 1}}{c_j}} }, \]
\[{\bar B_{M + 1}}(0) := \left\{ {u \in {H_{K,\lambda }}({\mathbb{R}^N})\ \big|\ ||u||_{K,\lambda} \leq M + 1} \right\}.\]
For a small $\mu>0$, we define
\[A_\mu ^\lambda  := \left\{ {u \in {{\bar B}_{M + 1}}(0)\ \big|\ ||u||_{K,\lambda,\mathbb{R}^N\setminus\Omega_\Gamma^{'}} \leq \mu ,{|\Phi _{\lambda ,j}}(u) - {c_j} |\leq \mu ,\forall j \in \Gamma } \right\}.\]
We also define
\[\Phi _{\lambda}^{c_\Gamma}:=\left\{u\in {H_{K,\lambda }}({\mathbb{R}^N})\ |\ \Phi_\lambda(u)\leq c_\Gamma\right\},\]
\[w=\sum\limits_{j = 1}^lw_j\in A_\mu ^\lambda \cap\Phi _{\lambda}^{c_\Gamma},\]
 which means $A_\mu ^\lambda \cap\Phi _{\lambda}^{c_\Gamma}\ne \emptyset$. Fix
\begin{equation}\label{e24}
0<\mu<\frac{1}{3}\text{min}\{c_j\ |\ j\in\Gamma\}.
\end{equation}
We obtain a uniform estimate for $||\Phi_\lambda^{\prime}(u)||_{K,\lambda}$ on $(A_{2\mu} ^\lambda\setminus A_\mu ^\lambda)\cap \Phi _{\lambda}^{c_\Gamma}$.

\begin{lemma}\label{l7}
Let $\mu>0$ satisfy \eqref{e24}. Then there exists $\sigma_0>0$ and $\Lambda_*\geq1$ independent of $\lambda$, such that
\[||\Phi_\lambda^{\prime}(u)||_{K,\lambda}\geq \sigma_0,\ \lambda\geq\Lambda_*,\forall\ u\in (A_{2\mu} ^\lambda\setminus A_\mu ^\lambda)\cap \Phi _{\lambda}^{c_\Gamma}.\]
\end{lemma}

\begin{proof}
We use the proof strategy of Proposition 4.4 in \cite{c19} to prove this lemma.

Proof by contradiction. Suppose there exist $\lambda_n\to \infty$,
\[u\in  (A_{2\mu} ^\lambda\setminus A_\mu ^\lambda)\cap \Phi _{\lambda}^{c_\Gamma},\]
such that $||\Phi_{\lambda_n}^{\prime}(u_n)||_K\to0$.

Since $u_n \in A_{2\mu} ^{\lambda_n}$ and $\{||u_n||_{K,\lambda_n}\}$ is a bounded sequence, we can conclude that $\{\Phi_{\lambda_n}(u_n)\}$ is also bounded. Therefore, we can assume that
\[\Phi_{\lambda_n}(u_n)\to c \in (-\infty,c_\Gamma].\]

According to Lemma \ref{p3}, in $H_K^1(\mathbb{R^N})$ we have a subsequence $u_n\to u$, where $u\in H_K^1(\Omega_\Gamma)$ is a nonnegative solution of equation \eqref{e27}, satisfying
\begin{equation}\label{e28}
u_n\to u,\ u\in H_K^1(\mathbb{R^N}),
\end{equation}
\begin{equation}\label{e29}
\lambda_n\int_{\mathbb{R^N}}K(x)a(x)|u_n|^2\to 0,
\end{equation}
\begin{equation}\label{e30}
||u_n||_{K,\lambda_n,\mathbb{R^N}
\setminus{\Omega_\Gamma}}\to 0.
\end{equation}
Since $c_j$ is the a least-energy value for $I_j$, we have two cases:\\
(i) $I_{j}(u|_{\Omega_j})=c_j,\ \forall \ j \in \Gamma.$\\
(ii) $I_{j_0}(u|_{\Omega_{j_0}})=0.$ That is, there exists $j_0\in\Gamma$ such that $u|_{\Omega_{j_0}}\equiv 0.$

If (i) occurs, according to \eqref{e28},\eqref{e29}, and \eqref{e30}, it can be seen that when $n$ is large, $u_n \in A_{\mu}^{\lambda_n}$. This contradicts the assumption that $u \in (A_{2\mu}^{\lambda_n} \setminus A_\mu^{\lambda_n})$.

If (ii) occurs, according to \eqref{e28} and \eqref{e29}, we have
\[|\Phi_{\lambda_n,j_0}(u_n)-c_{j_0}|\to c_{j_0}\geq 3\mu.\]
This contradicts the assumption that $u \in (A_{2\mu}^{\lambda_n} \setminus A_\mu^{\lambda_n})$. Therefore, neither (i) nor (ii) holds. The proof is complete.
\end{proof}

\begin{lemma}\label{l9}
Let $\mu>0$ satisfy  \eqref{e24}, and let $\Lambda_*\geq1$ be a constant given in Lemma \ref{l7}. Then, for $\lambda \geq \Lambda_*$, there exists a positive solution $u_\lambda$ to equation \eqref{a1} in $A_\mu ^\lambda \cap\Phi _{\lambda}^{c_\Gamma}$.
\end{lemma}

\begin{proof}
Proof by contradiction. Suppose there is no critical point in $A_\mu ^\lambda \cap\Phi _{\lambda}^{c_\Gamma}$. Since $\Phi_\lambda$ satisfies the $\text{(PS)}$ condition in $\left(0, \left( {\frac{1}{2} - \frac{1}{{q + 1}}} \right)S^{N/2}\right)$, there exists a constant $d_\lambda>0$ such that
\[||\Phi_\lambda^{\prime}(u)||_K\geq d_\lambda, \ \forall\ u\in A_\mu ^\lambda \cap\Phi _{\lambda}^{c_\Gamma}.\]
From the assumption, we have
\[||\Phi_\lambda^{\prime}(u)||_K\geq \sigma_0, \ \forall\ u\in (A_{2\mu} ^\lambda\setminus A_\mu ^\lambda)\cap \Phi _{\lambda}^{c_\Gamma},\]
where $\sigma_0>0$ is independent of $\lambda$. We define $\Psi:H_{K,\lambda}^1(\mathbb{R}^N)\to \mathbb{R}$ and $W:\Phi _{\lambda}^{c_\Gamma} \to \mathbb{R}$ as continuous functions satisfying
\[\Psi (u) = \left\{ \begin{gathered}
  1,u \in  A_{3\mu/2} ^\lambda,\hfill \\
  0,u \notin A_{2\mu} ^\lambda, \hfill \\
\end{gathered}  \right.\]
\[0\leq\Psi (u)\leq 1, u\in H_{K,\lambda}^1(\mathbb{R}^N)\]
and
\[W (u) = \left\{ \begin{gathered}
  -\Psi(u)||Y(u)||_K^{-1}||Y(u)||_K,u \in  A_{2\mu} ^\lambda,\hfill \\
  0,u \notin A_{2\mu} ^\lambda, \hfill \\
\end{gathered}  \right.\]
where $Y$ is the pseudo-gradient vector field of $\Phi_{\lambda}$ on $N=\{u\in H_{K,\lambda}^1(\mathbb{R}^N):\Phi_{\lambda}^{\prime}(u)\ne0\}$. Thus, using the properties of $Y$ and $\Phi_{\lambda}$, we have the following inequality,
\[||W(u)||_K \leq 1,\ \forall \ \lambda  \geq {\Lambda _*},u \in \Phi _\lambda ^{c_\Gamma }.\]
Consider the deformation flow defined as $\eta:[0,\infty)\times \Phi _{\lambda}^{c_\Gamma} \to \Phi _{\lambda}^{c_\Gamma}$,
\[\left\{ \begin{gathered}
  \frac{d\eta}{dt}=W(\eta),\hfill \\
  \eta(0,u)=u\in\Phi _{\lambda}^{c_\Gamma}. \hfill \\
\end{gathered}  \right.\]
Note that there exists $K_*>0$ such that
\[|\Phi_{\lambda,j}(u)-\Phi_{\lambda,j}(v)|\leq K_*||u-v||_{K,\lambda,\Omega_j^{'}}, \ \forall\ u,v \in \bar{B}_{M+1}(0), \forall \ j\in \Gamma.\]
Using a similar argument as in \cite{c12}, we obtain two numbers $T=T(\lambda)>0$ and $\varepsilon_*>0$ independent of $\lambda$ such that
\[\gamma^*(s_1,s_2,\cdots,s_l)=\eta(T,\gamma_0(s_1,s_2,\cdots,s_l))\in\Gamma_*,\]
\[\mathop {\max }\limits_{({s_1}, \cdots ,{s_l}) \in {{[1/{R^2},1]}^l}} {\Phi _\lambda }(\gamma^* ({s_1},\cdots,s_l))\leq c_\Gamma-\varepsilon_*.\]
Combining the definition of $b_{\lambda,\Gamma}$ and the above conclusion, we obtain the inequality
\[b_{\lambda,\Gamma}\leq c_\Gamma-\varepsilon_*, \ \forall\ \lambda \geq \Lambda_*.\]
This contradicts Corollary \ref{co1}.
\end{proof}

Now we prove Theorem \ref{t1}.
\begin{proof}
According to Lemma \ref{l9}, there exists a family of positive solutions $u_\lambda$ to equation \eqref{a1} with the following properties:

(i) Fix $\mu>0$. There exists $\lambda^*$ such that
\[||u_\lambda||_{K,\lambda,\mathbb{R}^N\setminus\Omega_\Gamma^{'}}\leq \mu,\ \forall\ \lambda\geq \lambda^*.\]
Therefore, from the proof of Lemma \ref{l8}, by choosing $\mu$ sufficiently small, we can conclude that
\[|u_\lambda|_{K,\infty,\mathbb{R}^N\setminus\Omega_\Gamma^{'}}\leq e,\ \forall\ \lambda\geq \lambda^*.\]
This implies that $u_\lambda$ is a positive solution to equation \eqref{a1}.

(ii) Fix $\lambda_n\to\infty$ and $\mu_n\to0$. The sequence $\{u_{\lambda_n}\}$ satisfies
\[\Phi _{\lambda _n}(u_{\lambda_n})=0,\ \forall\ n \in \mathbb{N},\]
\[||u_{\lambda_n}||_{K,\lambda_n,\mathbb{R}^N\setminus\Omega_\Gamma}\to 0,\]
\[\Phi _{\lambda _n,j}(u_{\lambda_n})\to c_j,\ \forall \ j \in \Gamma,\]
\[u_{\lambda_n}\to u\in H_K^1(\mathbb{R}^N), u\in H_K^1(\Omega_\Gamma),\]
from which the proof of Theorem \ref{t1} follows.
\end{proof}


\end{document}